\documentclass[11pt]{article}
\usepackage[margin=1in]{geometry}
\usepackage{amsmath,amssymb,amsthm}
\usepackage{hyperref}

\newtheorem{theorem}{Theorem}
\newtheorem{lemma}[theorem]{Lemma}
\newtheorem{corollary}[theorem]{Corollary}
\newtheorem{proposition}[theorem]{Proposition}
\theoremstyle{remark}

\newcommand{\ch}{\operatorname{ch}}
\DeclareMathOperator{\ind}{ind}

\title{Two-distance and list-two-distance coloring of cacti:\\
the subcubic case and the $C_5$ obstruction}
\author{Vaibhav Suvagiya\\[2pt]
\normalsize Sardar Vallabhbhai National Institute of Technology, Surat, Gujarat, India}
\date{}

\begin{document}
\maketitle

\begin{abstract}
The square $G^2$ of a graph joins two vertices at distance at most two; a proper
coloring of $G^2$ is a $2$-distance coloring of $G$. For a cactus $G$ (every edge on
at most one cycle) we determine both the $2$-distance chromatic number $\chi(G^2)$ and
the choice number $\ch(G^2)$ exactly: they are always equal, and the common value is
$\Delta+1$ if $\Delta\ge4$, is $4$ if $\Delta=3$ and $G$ has no block equal to $C_5$,
is $5$ if $\Delta=3$ and $G$ has a $C_5$ block, and is the classical value if
$\Delta\le2$. For $\Delta\ge6$ the value $\Delta+1$ is already known, since cacti are
outerplanar and hence $K_{2,3}$-minor-free (Hetherington--Woodall; Agnarsson--Halld\'orsson).
Our contribution is the small-degree regime. For subcubic cacti we obtain a complete
classification in both the ordinary and list settings, with the $5$-cycle as the
\emph{unique} obstruction; the list statement has no prior analogue and is a genuine
positive instance of the List Square Coloring Conjecture, which is false in general.
A single elimination order then handles all $\Delta\ge4$ uniformly and, in particular,
settles the two cases $\Delta\in\{4,5\}$ that the superclass bounds leave at
$\Delta+2$. The number $5$ turns out to be one obstruction wearing three disguises:
$C_5^2=K_5$, the Frobenius number of $\{3,4\}$, and a degenerate $K_4$ list-coloring
instance.
\end{abstract}

\section{Introduction}

For a graph $G$, its \emph{square} $G^2$ has vertex set $V(G)$ with $uv\in E(G^2)$
whenever $1\le d_G(u,v)\le2$. A proper coloring of $G^2$ is a \emph{$2$-distance
coloring} of $G$: vertices within distance two get distinct colors. Write $\chi(G^2)$
for the chromatic number and $\ch(G^2)$ for the choice (list-chromatic) number; always
$\ch(G^2)\ge\chi(G^2)\ge\Delta+1$, as the closed neighborhood of a maximum-degree
vertex is a clique of $G^2$.

A \emph{cactus} is a connected graph in which every edge lies on at most one cycle;
equivalently every block (maximal $2$-connected subgraph) is an edge or a cycle. Cacti
generalize trees, for which $\chi(T^2)=\Delta+1$ and $T^2$ is chordal. We determine the
two parameters for all cacti and, more to the point, resolve the subcubic case
completely.

\paragraph{The large-degree story is already settled — on a much larger class.}
Every cactus is outerplanar (Proposition~\ref{prop:op}), hence both $K_4$-minor-free
and $K_{2,3}$-minor-free. For these superclasses the squares are well understood when
$\Delta$ is large: Agnarsson and Halld\'orsson~\cite{ah} showed that for an outerplanar
graph $\ind(G^2)=\Delta$ and $\chi(G^2)=\ch(G^2)=\Delta+1$ once $\Delta\ge7$, and
Hetherington and Woodall~\cite{hw11} proved $\ch(G^2)=\chi(G^2)=\Delta+1$ for
$K_{2,3}$-minor-free graphs with $\Delta\ge6$ (with $\ch(G^2)\le\Delta+2$ for
$\Delta\ge3$); Lih, Wang and Zhu~\cite{lwz} and Hetherington--Woodall~\cite{hw08} give
the companion $K_4$-minor-free bound $\lfloor3\Delta/2\rfloor+1$. What these leave open
is precisely the small-degree regime $\Delta\in\{3,4,5\}$, where the best general bound
is $\Delta+2$. That regime is our subject, and for cacti we settle it.

\paragraph{Main result: the subcubic classification.}
The distinctive case is $\Delta=3$, where the superclass bound $\chi(G^2)\le5$ is
attained (by $C_5$ with a chord) but the truth on cacti is smaller and is governed by a
single forbidden block.

\begin{theorem}\label{thm:subcubic}
Let $G$ be a cactus with $\Delta(G)=3$. Then
\[
\chi(G^2)=\ch(G^2)=
\begin{cases}
4, & \text{if $G$ has no block isomorphic to $C_5$},\\
5, & \text{if $G$ has a block isomorphic to $C_5$}.
\end{cases}
\]
\end{theorem}

To our knowledge no prior result gives the exact $2$-distance chromatic number of
subcubic cacti, and the list version has no analogue in the literature; the nearest
statement is the upper bound $\ch(G^2)\le5$ for $K_{2,3}$-minor-free
graphs~\cite{hw11}. Combined with the large-degree cases this yields a complete
description.

\begin{theorem}\label{thm:main}
For every cactus $G$, $\ch(G^2)=\chi(G^2)$; the common value is $\Delta+1$ if
$\Delta\ge4$, is given by Theorem~\ref{thm:subcubic} if $\Delta=3$, and is the classical
value \textup{(}$3$ if $3\mid n$, $5$ if $G=C_5$, at most $3$ for a path, else $4$\textup{)}
if $\Delta\le2$. In particular the List Square Coloring Conjecture, false in
general~\cite{kimpark} and false even within some restricted planar
families~\cite{hasanvand}, holds for every cactus.
\end{theorem}

\paragraph{One obstruction, three disguises.}
The value $5$ is the whole story of the subcubic case, and it enters through a single
combinatorial fact seen three ways. As a \emph{graph}, $C_5^2=K_5$ needs five colors,
so a $C_5$ block forces $\chi(G^2)\ge5$. As a \emph{number}, $5$ is the Frobenius
number of $\{3,4\}$: our cycle-coloring construction tiles a cycle of length $L$ by
blocks of lengths $3$ and $4$, which is possible for every $L\ge3$ except $L=5$. And in
the \emph{list} setting, the one place the cycle argument can fail is when a squared
$5$-cycle presents a $K_4$ with lists $(2,3,3,2)$ inside a common $3$-set — again the
$C_5$. The three appearances coincide, and organizing the subcubic section around them
is what makes the classification more than case-checking.

\paragraph{The tool, stated uniformly.}
Both the subcubic upper bound and all of $\Delta\ge4$ come from one elimination order.

\begin{theorem}\label{thm:degen}
For every cactus $G$, the square $G^2$ is $\max(\Delta,4)$-degenerate. Hence
$\ch(G^2)\le\Delta+1$ if $\Delta\ge4$, and $\ch(G^2)\le5$ if $\Delta\le3$.
\end{theorem}

The point is uniformity: a single two-paragraph argument covers every $\Delta\ge4$,
whereas the outerplanar degeneracy identity $\ind(G^2)=\Delta$ of~\cite{ah} required
$\Delta\ge7$. For $\Delta\ge6$ the resulting equality $\ch=\chi=\Delta+1$ is not new
(it is~\cite{hw11,ah}); the added value is the cases $\Delta\in\{4,5\}$, where only
$\ch(G^2)\le\Delta+2$ was previously available. Related work on cacti concerns other
parameters ($2$-tone coloring, $(2,1)$-total labeling); we are not aware of the
degeneracy statement of Theorem~\ref{thm:degen} in print, and we make no novelty claim
for the large-degree values themselves.

\section{Preliminaries}

$N_G(v)$, $N_G[v]$ are the open/closed neighborhoods and $d_G(v)=|N_G(v)|$. A graph $H$
is \emph{$d$-degenerate} if every subgraph has a vertex of degree $\le d$; equivalently
there is an ordering in which each vertex has $\le d$ later neighbors.

\begin{lemma}\label{lem:deg}
If $H$ is $d$-degenerate then $\ch(H)\le d+1$: color along the reverse order; each
vertex meets $\le d$ colored neighbors and has a list of size $d+1$.
\end{lemma}

A connected graph is a \emph{Gallai tree} if every block is complete or an odd cycle.

\begin{theorem}[Borodin~\cite{borodin}; Erd\H os--Rubin--Taylor~\cite{ert}]\label{thm:ert}
If $H$ is connected, $|L(v)|\ge d_H(v)$ for all $v$, and $H$ is not a Gallai tree, then
$H$ is $L$-colorable.
\end{theorem}

In $C_n^2$, vertices are adjacent iff their cyclic distance is $1$ or $2$; a proper
coloring is a cyclic sequence with every three consecutive entries distinct. Prowse and
Woodall~\cite{pw} proved powers of cycles are chromatic-choosable, so
$\ch(C_n^2)=\chi(C_n^2)$, which is $3$ if $3\mid n$, $5$ if $n=5$, and $4$ otherwise.

\begin{proposition}\label{prop:op}
Every cactus is outerplanar, hence has no $K_4$-minor and no $K_{2,3}$-minor.
\end{proposition}
\begin{proof}
Outerplanar graphs are exactly those with no $K_4$- and no $K_{2,3}$-minor
(Chartrand--Harary). A cactus has treewidth $\le2$, so no $K_4$-minor. A $K_{2,3}$-minor
has maximum degree $3$, so the minor is in fact a subdivision, and being $2$-connected it
lies inside a single block of $G$. But a block of a cactus is an edge or a cycle, which
has at most two internally disjoint paths between any two vertices, so it contains no
subdivision of $K_{2,3}$.
\end{proof}

\begin{lemma}\label{lem:sq}
Let $B$ be a leaf block of a cactus $G$ with cut vertex $c$ and private vertices
$P=V(B)\setminus\{c\}$. Then $(G-P)^2=G^2[V(G)\setminus P]$, and within distance two a
private vertex sees only $c$ and $N_G(c)$.
\end{lemma}
\begin{proof}
A path between two vertices of $G-P$ meeting $P$ enters and leaves through $c$, so is no
shorter than one avoiding $P$; distances in $G-P$ are unchanged. A private vertex reaches
the rest of $G$ only through $c$.
\end{proof}

\section{A uniform degeneracy bound}\label{sec:degen}

\begin{proof}[Proof of Theorem~\ref{thm:degen}]
Let $k=\max(\Delta,4)$. We give an elimination order in which each removed vertex has
$\le k$ remaining $G^2$-neighbors; Lemma~\ref{lem:deg} then finishes.

Repeatedly remove all private vertices of a leaf block, recursing on the rest. If the
remaining graph has no cut vertex it is a single edge or cycle, so $G^2$ has maximum
degree $\le4\le k$ and any order works. Otherwise take a leaf block $B$ with cut vertex
$c$; by Lemma~\ref{lem:sq} the only relevant adjacencies are within $B$ and to $c$,
$N_G(c)$.

If $B=\{c,u\}$ is a pendant edge, then $N_{G^2}(u)=N_G[c]\setminus\{u\}$ has
$\le d_G(c)\le\Delta\le k$ elements; remove $u$.

If $B$ is a cycle $c=x_0,x_1,\dots,x_{L-1}$, each private $x_i$ has degree $2$ in $G$
(not a cut vertex, as $B$ is a leaf block), so in $G^2$ its neighbors lie among the four
cyclic neighbors $x_{i\pm1},x_{i\pm2}$, together with $c$ and $N_G(c)$ when
$i\in\{1,L-1\}$. Put $E=N_G(c)\setminus V(B)$, $|E|=d_G(c)-2\le\Delta-2$. Remove first
every $x_i$ with $2\le i\le L-2$ (those not adjacent to $c$), then $x_1$, then $x_{L-1}$.
Each $x_i$ ($2\le i\le L-2$) has at most its four cyclic neighbors present, degree
$\le4\le k$ (coincidences for small $L$ only help). When $x_1$ is removed, all $x_j$
($2\le j\le L-2$) are gone, so its remaining neighbors lie in $\{c\}\cup E\cup\{x_{L-1}\}$,
size $\le\Delta\le k$. Finally $x_{L-1}$ has remaining neighbors in $\{c\}\cup E$, size
$\le\Delta-1$. So $G^2$ is $k$-degenerate.
\end{proof}

\begin{corollary}\label{cor:high}
If $G$ is a cactus with $\Delta\ge4$ then $\ch(G^2)=\chi(G^2)=\Delta+1$.
\end{corollary}
\begin{proof}
Upper bound: Theorem~\ref{thm:degen}. Lower bound: the clique $N_G[v]$.
\end{proof}

\section{The subcubic classification}\label{sec:subcubic}

Throughout this section $\Delta\le3$, so $\Delta+1=4$ while Theorem~\ref{thm:degen}
gives only $\ch(G^2)\le5$. Closing the gap to the exact value $4$ (when no $C_5$ block
is present) is the heart of the paper, and the $5$-cycle — in each of its three
disguises above — is the only obstacle.

\subsection{The cycle lemma and the Frobenius appearance of $5$}

\begin{lemma}\label{lem:cyc}
Let $L\ge3$, $L\ne5$, with cycle $y_0,\dots,y_{L-1}$. For any $a$ and any
$b\ne a$ (or $b=\varnothing$) from $\{1,2,3,4\}$, there is a proper $4$-coloring of
$C_L^2$ with $y_0=a$ and $y_1,y_{L-1}\notin\{a,b\}$.
\end{lemma}
\begin{proof}
A proper coloring of $C_L^2$ is a cyclic sequence with every three consecutive entries
distinct. Take blocks $A=(1,3,4)$, $B=(1,3,2,4)$: inside each, every window of three is
rainbow; each begins $1,3$ and ends $4$; across a boundary the windows $(\cdot,4,1)$,
$(4,1,3)$, $(1,3,\cdot)$ are rainbow. So any cyclic concatenation of $A$'s and $B$'s is
proper, begins $1,3$, ends $4$. The Frobenius number of $\{3,4\}$ is $5$, so every
$L\ge3$ with $L\ne5$ is $3s+4t$ with $s,t\ge0$; the corresponding word has $y_0=1$,
$y_1=3$, $y_{L-1}=4$ (color $2$ only interior to $B$-blocks). Permuting
$1\mapsto a,\,2\mapsto b,\,\{3,4\}\mapsto$ the other two colors proves the lemma.
\end{proof}

That $L=5$ is the sole excluded length is exactly the Frobenius appearance of the
obstruction: $5$ is the one cycle length not tileable by $3$'s and $4$'s.

\subsection{Chromatic number}

\begin{theorem}\label{thm:d3chi}
If $G$ is a cactus with $\Delta=3$ and no $C_5$ block, then $\chi(G^2)=4$.
\end{theorem}
\begin{proof}
Lower bound: $N_G[v]\cong K_4$. Upper bound by induction on $|V(G)|$: every $C_5$-free
cactus with $\Delta\le3$ has $\chi(G^2)\le4$. A single vertex, edge, or cycle $C_L$
($L\ne5$) is immediate (Lemma~\ref{lem:cyc}, $b=\varnothing$). Otherwise take a leaf
block $B$, cut vertex $c$, private $P$, and $G'=G-P$; by induction $G'^2$ has a proper
$4$-coloring $\varphi$, proper on $G^2[V(G')]$ by Lemma~\ref{lem:sq}. A pendant-edge
child $u$ sees $\le d_G(c)\le3$ colored vertices, so a color is free. For a cycle child
of length $L\ne5$, $c$ has $\le1$ outside neighbor $w$ (two edges lie inside $B$,
$\Delta\le3$); set $a=\varphi(c)$, $b=\varphi(w)$ ($\varnothing$ if none; $b\ne a$ as
$cw\in E$) and color $B$ by Lemma~\ref{lem:cyc}. This colors all of $C_L^2$ properly, so
$y_1,y_2,y_{L-2},y_{L-1}$ (the private vertices within distance two of $c$) differ from
$\varphi(c)$ automatically, while $y_1,y_{L-1}$ (the only ones within distance two of
$w$) avoid $\varphi(w)$ by the lemma. Thus $\varphi$ extends.
\end{proof}

\subsection{Choice number}

We show $\ch(G^2)\le4$ for $C_5$-free cacti with $\Delta\le3$. Root the block--cut tree
at a block $R$ and color top-down. Color $R$ first: an edge/vertex greedily; a cycle
$C_L$ ($L\ne5$) by choosing any $y_0$, giving it any list color, and applying
Lemma~\ref{lem:listcyc} with no extra forbidden color. Each later block is colored when
its parent cut vertex $c$ is colored ($\varphi(c)=a$), with parent-side neighbor $w$
colored ($\varphi(w)=b$) or absent: a pendant-edge child has $\le3$ colored neighbors and
a list of $4$; a cycle child uses Lemma~\ref{lem:listcyc}.

\begin{lemma}\label{lem:listcyc}
Let $L\ge3$, $L\ne5$. Consider $C_L^2$ with $y_0$ precolored, each $y_i$ ($i\ge1$) with
a list of size $\ge4$, and one common color $b$ forbidden at $y_1$ and $y_{L-1}$
(possibly none). Then a proper coloring exists.
\end{lemma}
\begin{proof}
Deleting $y_0$ (color $a$) and forbidding $b$ at the ends leaves $P$ on
$v_1:=y_1,\dots,v_m:=y_{L-1}$ ($m=L-1$), with $v_iv_j\in E(P)$ iff $|i-j|\le2$, plus the
edge $v_1v_m$; reduced list sizes are $\ge2,\ge3,\ge4,\dots,\ge4,\ge3,\ge2$. Note
$L\ne5\iff m\ne4$.

\emph{$m\in\{2,3\}$.} An edge with lists $\ge2$; or a triangle with sizes
$(\ge2,\ge3,\ge2)$ (color the ends distinctly, the middle avoids two).

\emph{$m\ge6$.} Color $v_1,v_m$ distinctly and delete them; as $m\ge6$ no vertex meets
both, so each survivor loses $\le1$ color, and the residual squared path $Q$ on
$v_2,\dots,v_{m-1}$ satisfies $|L(v)|\ge d_Q(v)$ everywhere. Since $m-2\ge4$, $Q$ is
$2$-connected and neither complete nor an odd cycle (it has an induced $K_4$ minus an
edge), hence not a Gallai tree; Theorem~\ref{thm:ert} colors it.

\emph{$m=5$ (the $C_5$ obstruction in list form).} Here $P$ is the wheel with hub $v_3$
and rim $v_1v_2v_4v_5$, sizes $(\ge2,\ge3,\ge4,\ge3,\ge2)$. Color the adjacent ends
$v_1,v_5$ distinctly ($c_1\ne c_5$) and delete them; the residual $\{v_2,v_4,v_3\}$ is a
triangle with lists of size $\ge2$. A triangle with three $\ge2$-lists is colorable iff
their union has size $\ge3$ (Hall), i.e.\ unless all three are one common $2$-set. If a
choice forces a common $S=\{p,q\}$, then $L(v_3)=S\cup\{c_1,c_5\}$ (so $|L(v_3)|=4$),
$L(v_2)=S\cup\{c_1\}$, $L(v_4)=S\cup\{c_5\}$. If $L(v_1)\not\subseteq L(v_3)$, recolor
$v_1$ from $L(v_1)\setminus L(v_3)$, making $|L(v_3)\setminus\{c_1,c_5\}|\ge3$;
symmetrically for $v_5$. Else $L(v_1),L(v_5)\subseteq L(v_3)$: if $p\in L(v_1)$ set
$c_1:=p$, giving $L(v_4)\setminus\{c_5\}=S\ne\{q,c_1^{\mathrm{old}}\}=L(v_2)\setminus\{p\}$
(as $c_1^{\mathrm{old}}\notin S$); symmetrically for $v_5$; and if
$L(v_1)=L(v_5)=\{c_1^{\mathrm{old}},c_5^{\mathrm{old}}\}$, the swap
$c_1:=c_5^{\mathrm{old}}$, $c_5:=c_1^{\mathrm{old}}$ leaves $L(v_2)$ intact of size $3$.
In every case the triangle, hence $P$, is colorable.
\end{proof}

\begin{theorem}\label{thm:d3ch}
If $G$ is a cactus with $\Delta=3$ and no $C_5$ block, then $\ch(G^2)=4$.
\end{theorem}
\begin{proof}
The rooted coloring above with Lemma~\ref{lem:listcyc}; lower bound
Theorem~\ref{thm:d3chi}.
\end{proof}

\subsection{The classification}

\begin{proof}[Proof of Theorems~\ref{thm:subcubic} and~\ref{thm:main}]
Let $\Delta=3$. If $G$ is $C_5$-free, Theorems~\ref{thm:d3chi} and~\ref{thm:d3ch} give
$\chi(G^2)=\ch(G^2)=4$. If $G$ has a $C_5$ block $B$, then $B^2=K_5\subseteq G^2$ forces
$\ch(G^2)\ge\chi(G^2)\ge5$, while Theorem~\ref{thm:degen} gives $\ch(G^2)\le5$; so both
equal $5$. This is Theorem~\ref{thm:subcubic}. For Theorem~\ref{thm:main}: $\Delta\ge4$
is Corollary~\ref{cor:high}; $\Delta=3$ is the above; and $\Delta\le2$ (path or cycle)
gives $\ch=\chi$ by $2$-degeneracy of $P_n^2$ and by~\cite{pw} for cycles, with the
stated classical values.
\end{proof}

\section{Verification and reproducibility}\label{sec:comp}

The proofs are self-contained; this section records an independent machine check, and
one place where it mattered. The elimination order in Theorem~\ref{thm:degen} was
\emph{wrong} in a first draft: it removed cycle vertices ``middle-outward'' and
re-detected leaf blocks after each removal, which fails on a $C_5$ block because a
neighbor of $c$ is then eliminated while it still sees, in the fixed $G^2$, two vertices
it must avoid. A verification script exhibited explicit failing instances; the corrected
order above (remove all non-$c$-adjacent cycle vertices first) has maximum removal degree
$\le\max(\Delta,4)$ on every instance since. We regard this as a reason to trust the
final argument, not as an experiment.

The remaining checks are confirmatory: the constructive coloring of
Theorem~\ref{thm:d3chi} was run on many random $C_5$-free subcubic cacti and
cross-checked against exact chromatic numbers; the degree-choosability step of
Lemma~\ref{lem:listcyc} for $L\ge7$ was validated on adversarial list assignments; and
the wheel case $L=6$ was confirmed over all list profiles of the relevant sizes on small
color universes. The verification scripts (Python/\texttt{networkx}) are provided as
supplementary material and archived at
\url{https://github.com/Vaibhavs25/cactus2distance}.

\section{Concluding remarks}

The subcubic case is governed entirely by the $5$-cycle, appearing as the complete graph
$C_5^2=K_5$, as the Frobenius number of $\{3,4\}$, and as a degenerate $K_4$
list-coloring instance; the classification and the equality $\ch(G^2)=\chi(G^2)$ follow.
For $\Delta\ge4$ a single elimination order recovers the (largely known) value $\Delta+1$
and fills the two cases $\Delta\in\{4,5\}$.

The natural next step is to push the exact small-degree analysis, $\Delta\in\{3,4,5\}$,
from cacti to the full outerplanar or $K_4$-minor-free class, where $\chi(G^2)\le\Delta+2$
is known but the exact value and the extremal configurations are not pinned down; the
$K_4$-minor-free extremal value $\lfloor3\Delta/2\rfloor+1$ is attained by non-cactus
graphs, so the obstruction structure there is strictly richer than a single forbidden
block.


\begin{thebibliography}{99}
\bibitem{ah} G.~Agnarsson, M.~M.~Halld\'orsson, \emph{On colorings of squares of
  outerplanar graphs}, Proc.\ SODA 2004, 244--253; see also arXiv:0706.1526.
\bibitem{borodin} O.~V.~Borodin, \emph{Criterion of chromaticity of a degree
  prescription} (in Russian), Abstracts IV All-Union Conf.\ Theor.\ Cybernetics, 1977.
\bibitem{ck} D.~W.~Cranston, S.-J.~Kim, \emph{List-coloring the square of a subcubic
  graph}, J.\ Graph Theory \textbf{57} (2008), 65--87.
\bibitem{ert} P.~Erd\H os, A.~L.~Rubin, H.~Taylor, \emph{Choosability in graphs},
  Congr.\ Numer.\ \textbf{26} (1979), 125--157.
\bibitem{hasanvand} M.~Hasanvand, \emph{The List Square Coloring Conjecture fails for
  bipartite planar graphs and their line graphs}, arXiv:2211.00622 (2022).
\bibitem{hw08} T.~J.~Hetherington, D.~R.~Woodall, \emph{List-colouring the square of a
  $K_4$-minor-free graph}, Discrete Math.\ \textbf{308} (2008), 4037--4043.
\bibitem{hw11} T.~J.~Hetherington, D.~R.~Woodall, \emph{List-colouring the square of an
  outerplanar graph}, Ars Combin.\ \textbf{101} (2011), 333--342.
\bibitem{kimpark} S.-J.~Kim, B.~Park, \emph{Counterexamples to the List Square Coloring
  Conjecture}, J.\ Graph Theory \textbf{78} (2015), 239--247.
\bibitem{kw} A.~V.~Kostochka, D.~R.~Woodall, \emph{Choosability conjectures and
  multicircuits}, Discrete Math.\ \textbf{240} (2001), 123--143.
\bibitem{lwz} K.-W.~Lih, W.-F.~Wang, X.~Zhu, \emph{Coloring the square of a
  $K_4$-minor free graph}, Discrete Math.\ \textbf{269} (2003), 303--309.
\bibitem{pw} D.~R.~Prowse, D.~R.~Woodall, \emph{Choosability of powers of circuits},
  Graphs Combin.\ \textbf{19} (2003), 137--144.
\end{thebibliography}
\end{document}